\documentclass[proceedings]{aofa}
\usepackage[round]{natbib}
\usepackage{graphicx,tikz}

\usepackage{subfigure}
\usepackage{amssymb,amsmath,latexsym,times}

\usepackage[pdftex]{hyperref}
\hypersetup{plainpages=True, pdfstartview=FitV,
	colorlinks=true,linkcolor=blue,citecolor=blue}
\newtheorem{thm}{Theorem}
\newtheorem{cor}[thm]{Corollary}
\newtheorem{lem}[thm]{Lemma}
\newtheorem{prop}[thm]{Proposition}

\author{Vytas Zacharovas}
\title{On the exponential decay   of the characteristic function of the quicksort distribution}%

\address{	Department of Mathematics and Informatics\\
	Vilnius University\\Naugarduko 24\\Vilnius, Lithuania\\E-mail: vytas.zacharovas@mif.vu.lt}
\begin{document}
%

%\thanks{dddd}%
%\subjclass{}%
%\keywords{}%

%\date{}%
%\dedicatory{}%
%\commby{}%
% ----------------------------------------------------------------

\maketitle
\begin{abstract}
	\paragraph{Abstract.}
	We prove that the characteristic function of the quicksort distribution is exponentially decreasing at infinity. As a consequence it follows that the density of the quicksort distribution can be analytically extended to the vicinity of the real line.
\end{abstract}

% ----------------------------------------------------------------
\section{Introduction}

Let $X_n$ be the number of steps required by Quicksort algorithm to sort the list of values \(\sigma(1),\sigma(2),\ldots,\sigma(n)\) where \(\sigma\) is a random permutation chosen with uniform probability from the set of all permutations \(S_n\) of order \(n\).
It has been proven by \cite{regnier_1989} and  \cite{rosler} that
the appropriately  scaled distribution of \(X_n\) converges to some limit law
\[
\frac{X_n-\mathbb{E}X_n}{n}\to^d Y
\]
as \(n\to \infty\).
Let us denote as \(f(t)\) the characteristic function of the limiting distribution 
\[
f(t)=\mathbb{E}e^{it Y}
\]
\cite{tan_hadjicostas} proved that the characteristic
function $f(t)$ has a density $p(x)$.  \cite{knessl_szpankowski} using heuristic approach established a number of very precise estimates for the behavior of \(p(x)\) at infinity. Later 
\cite{fill_janson_2000} showed that the characteristic function
$f(t)$ of the limit quicksort distribution together with its all
derivatives decrease faster than any polynomial at infinity. More
precisely they showed that for all  real $p>0$ there is such a
constant $c_p$ that
\[
|f(t)|\leqslant \frac{c_p}{|t|^p},\quad \hbox{for all} \quad t\in
\mathbb{R}.
\]
They also proved that
\[
c_p\leqslant 2^{p^2+6p}.
\]
Hence
\[
|f(t)|\leqslant \inf_{p>0}\frac{2^{p^2+6p}}{|t|^p}.
\]
The infimum in the above inequality can be evaluated as
\[
|f(t)|\leqslant
\inf_{p>0}\frac{2^{p^2+6p}}{|t|^p}\leqslant |t|^3 e^{-\frac{\log^2
|t|}{4\log 2}}.
\]

The main result of this paper is the following theorem stating that the characteristic function \(f(t)\) of limiting Quicksort distribution decreases exponentially at infinity.

\begin{thm}
	\label{main_thm}
There is a constant $\eta>0$ such that 
\[
f(t)=O(e^{-\eta|t|})
\]
as $|t|\to \infty$. 
\end{thm}
 \begin{cor}
 	\label{cor}
 	Quicksort distribution has a bounded density that can be extended
 	analytically to the vicinity of the real line $|\Im(s)|<\eta$. Where \(\eta\) is the same positive number as in the formulation of Theorem \ref{main_thm}.
 \end{cor}

% As a consequence we conclude that density $p(x)$
%can be continued analytically into a vicinity of the real line. Fill
%and Janson \cite{fill_janson_2001} proved that if $y\geqslant 302.1$
%then
%\[
%\mathbb{P}(Y\geqslant y)\leqslant e^{-y(\log y -1-\log 2)}
%\]
%and for the left tail they obtained
%\[
%P(Y\leqslant y)\leqslant e^{-y^2/5}, \quad \hbox{if}\quad y\leqslant 0.
%\]
%We obtain an improvement of the estimate of the left tail
%\[
%\mathbb{P}(Y\leqslant y) \ll
%e^{-2e^{\frac{|y|}{2}-1}-\frac{|y|}{4}}
%\]
%if $y\leq 0$.  The above estimate is consistent with estimates
%of the behavior of the density $p(x)$ at infinity that were conjectured
%by Knessl and Szpankowski \cite{knessl_szpankowski}  using the method of
%matched asymptotics.

\section{Proofs}
It has been shown in \cite{rosler} that the characteristic function \(f(t)\) satisfies the functional equation
\[
f(t)=e^{it}\int_0^1f(tx)f\bigl(t(1-x)\bigr)e^{2itx\log
	x+2it(1-x)\log(1-x)}\,dx
\]
which after a change of variables $x\to y/t$ becomes
\[
tf(t)e^{2it\log t}=e^{it}\int_0^tf(y)f(t-y)e^{2iy\log
	y+2i(t-y)\log(t-y)}\,dy
\]
It follows hence by taking Laplace transform of the both sides that
function
\[
\psi(s)=\int_0^\infty f(t)e^{2it\log t}e^{-st}\,dt
\]
satisfies an equation
\begin{equation}
\label{shift_diff_eq}
-\psi'(s)=\psi^2(s-i).
\end{equation}
The Laplace transform \(\psi(s)\) together with the above differential equation will be the main tool of proving the result stated in the introduction.

It is well known that the quicksort distribution has finite moments
of all orders. In the following analysis we will only need the fact
that it has finite first moment, which implies that $|f'(t)|$ is
bounded. Thus integrating by parts we conclude that
\begin{equation}
\label{upper_bound_for_psi}
\begin{split}
\psi(s)&=\int_0^\infty f(t)e^{2it\log t}
e^{-st}\,dt
\\
&=\frac{1}{s}+\frac{1}{s}\int_0^\infty\bigl( f'(t)e^{2it\log
t}+f(t)e^{2it\log t}(2i\log t+2i)\bigr) e^{-st}\,dt
\\
&\leqslant \frac{A}{|s|}\left(1+\frac{|\log \Re s|}{\Re s}\right),
\end{split}
\end{equation}
for all \(s\) lying in the right half-plane $\Re s>0$ and \(A>0\) being some positive absolute constant.
\begin{lem}
	\label{lem_ineq} For all \(s\) lying in the right half-plane $\Re s>0$  and all integer \(n\geqslant 0\) holds the inequality
\[
\bigl|\psi^{(n)}(s)\bigr|\leqslant n!\left(\max_{
r\in \{0,1,\ldots,n\}}\bigl|\psi(s-ir)\bigr|\right)^{n+1}
\]
\end{lem}
\begin{proof}
The proof is done by applying mathematical induction on \(n\) and using the fact that the differential equation for $\psi(s)$ allows us to express
the derivatives  $\psi^{(n)}(s)$ as a polynomial function of
$\psi(s-ik)$ with $0\leqslant k\leqslant n$.

Indeed, for \(n=0\) the above inequality becomes an identity. Suppose this identity holds for all \(n\) not exceeding \(m\). Let us consider now \(n=m+1\). Replacing the first derivative of \(\psi(s)\) by \(-\psi^2(s-i)\) we obtain
\[
\begin{split}
\psi^{(m+1)}(s)&=\bigl(\psi'(s)\bigr)^{(m)}=-\bigl(\psi^2(s-i)\bigr)^{(m)}
\\
&=-\sum_{k=0}^{m}\binom{m}{k}\psi^{(k)}(s-i)\psi^{(m-k)}(s-i).
\end{split}
\]
Thus applying the inductive hypothesis to the derivatives of \(\psi(s-i)\) we get
\[\begin{split}
\bigl|\psi^{(m+1)}(s)\bigr|&\leqslant \sum_{k=0}^{m}\binom{m}{k}k!\left(\max_{
	r\in \{0,1,\ldots,k\}}\bigl|\psi(s-i-ir)\bigr|\right)^{k+1}
(m-k)!\left(\max_{
	r\in \{0,1,\ldots,m-k\}}\bigl|\psi(s-i-ir)\bigr|\right)^{m-k+1}
\\
&\leqslant (m+1)!\left(\max_{
	r\in \{0,1,\ldots,n\}}\bigl|\psi(s-ir)\bigr|\right)^{m+2}.
\end{split}
\]
The last inequality is the same as stated in the lemma with \(n=m+1\). This completes the proof of the lemma
\end{proof}
\begin{lem} 
	\label{lem_psi}For all \(s\) lying in the lower part of the right half-plane \(\Re s>0\) and \(\Im s<0\) holds the inequality
	
\[
\bigl|\psi^{(n)}(s)\bigr|\leqslant n!\left(\frac{C(\sigma)}{|s|}\right)^{n+1}
\]
Where \(\sigma=\Re s\) and
\[
C(\sigma)=A\left(1+\frac{|\log \sigma|}{\sigma}\right)
\]
with some absolute constant \(A>0\).
\end{lem}
\begin{proof} 
	Our upper bound (\ref{upper_bound_for_psi}) for $\psi(s)$
	implies that for \(\Re s>0\) and \(\Im s<0\)  we have
	\[
	\max_{
		r\in \{0,1,\ldots,n\}}\bigl|\psi(s-ir)\bigr|\leqslant \max_{
		r\in \{0,1,\ldots,n\}} \frac{C(\sigma)}{|s-ir|} \leqslant \frac{C(\sigma)}{|s|}.
	\]
Since imaginary part of \(s\)	is negative so \(|s-ir|\geqslant |s|\). Using this inequality to evaluate the right hand side of the inequality of Lemma \ref{lem_ineq} we complete the proof of the lemma.

\end{proof}
\begin{prop}
The function $\psi(s)$ can be continued analytically to the whole
complex plane. Moreover, for all \(s\)  belonging to the lower half-plane \(\Im(s)<0\)  and \(\Re s\geqslant -B\) with any fixed \(B>0\) holds the estimate
\[
\psi(s)=O_B(1/|s|)
\]

\end{prop}
\begin{proof}
For \(\Re(s)\geqslant 1\) the estimate of the proposition already follows from (\ref{upper_bound_for_psi}).
By this estimate of Lemma \ref{lem_psi}  we have that the Taylor
series
\[
\psi(s)=\sum_{j=0}^\infty\frac{\psi^{(j)}(1-iK)}{j!}\bigl(s-(1-iK)\bigr)^j
\]
converges in the circle $|1-iK-s|<|1-iK|/C(1)$ and moreover in this
circle holds the estimate
\[
|\psi(s)|\leqslant
\sum_{j=0}^\infty\left(\frac{C(1)}{|1-iK|}\right)^{j+1}\bigl|s-(1-iK)\bigr|^j
=\frac{C(1)}{|1-iK|}\frac{1}{1-\frac{C(1)}{|1-iK|}|1-iK-s|}.
\]
This means that $\psi(s)$ can be analytically continued to the
region of complex plane that consists of such \(s\) that are contained in any of  the circles of radius $|1-iK|/C(1)$ with center at \(1-iK\) with some \(K>0\). Note that all complex number \(s\) with negative imaginary part such that $1+\frac{\Im (s)}{C(1)}\leqslant \Re(s)$ satisfy such condition. See the figure \ref{fig:continuation}.
\begin{center}
	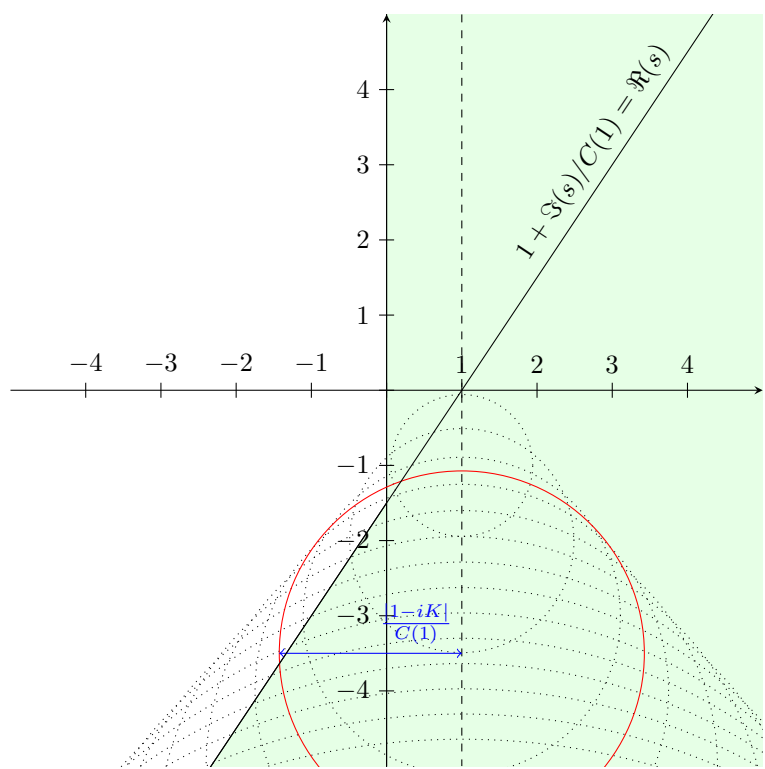
\begin{figure}[htbp]

\begin{tikzpicture}
	\usetikzlibrary{arrows,patterns}
%	\usetikzlibrary{math}
	\tikzstyle{myedgestyle} = [-stealth]
\clip(-5,5) rectangle (5,-5);

\def \c{1.5}

\draw[fill=green!10] (1-4,0-4*\c) -- (5,4*\c)--(5,-0-4*\c)--cycle;
\draw[fill=green!10]   (0,6) rectangle (6,-6);
\foreach \x in {1,2,3,4,5,6,7,8,9,10,11,12,13,14}{\pgfmathparse{sqrt(1+\x^2)/\c}\let\a\pgfmathresult;\draw[dotted]  (1,-\x ) circle (\a);}
\draw [-stealth](-5,0) -- (5,0);
\draw [-stealth](0,-5) -- (0,5);
\foreach \x in {-4,-3,-2,-1,0,1,2,3,4}{\draw (\x,-0.1) -- (\x,0.1);}
\foreach \x in {-4,-3,-2,-1,0,1,2,3,4}{\draw (-0.1,\x) -- (0.1,\x);}
\foreach \x in {-4,-3,-2,-1,1,2,3,4}{\draw (\x,0.1)node[above]{$\x$};}
\foreach \x in {-4,-3,-2,-1,1,2,3,4}{\draw (-0.1,\x)node[left]{$\x$};}

\draw[dashed] (1,5) -- (1,-5);

\def\x{3.5}
\pgfmathparse{sqrt(1+\x^2)/\c}\let\a\pgfmathresult
%\tikzmath{\a=sqrt(1+\x^2)/\c;};
\draw[color=red]  (1,-\x ) circle (\a);
\draw[<->,blue] (1,-\x) -- (1-\a,-\x)node[near start,above,blue]{$\frac{|1-iK|}{C(1)}$};
\draw[black] (1-4,0-4*\c) -- (5,4*\c)node[near end,above,sloped]{$1+\Im(s)/C(1)=\Re(s)$};

\end{tikzpicture}
\caption{The continuation of \(\psi(s)\) to the left half-plane}
  \label{fig:continuation}
\end{figure}
\end{center}
Note that $\psi(s)$ satisfies a shift-differential
equation (\ref{shift_diff_eq}) which is by integrating its both sides yields the  identity
\[
\psi(s)=\psi(s-i)+i\int_{0}^{1}\psi(s-i-it)^2\,dt.
\]
The repeated application of the above identity allows us to   continue \(\psi(s)\)  analytically to the whole
complex plane.

\end{proof}

We have already proven that for $\Im(s)\leqslant 0$ we have
\[
\psi(s)=O\left(\frac{1}{|s|}\right) 
\]
when $\Re(s)\geqslant -H$  with an arbitrary fixed $H>0$. Let us now
try to obtain a similar estimate for the values of $s$ lying in the
upper half-plane.

\begin{lem}
\label{bound_sup} For all $\sigma>0$ we have
\[
\sup_{y\in \mathbb{R}}|\psi(\sigma+iy)|< \frac{1}{\sigma}.
\]
\end{lem}
\begin{proof}  The proof of the lemma relies on a standard trick that is used to
	prove that if a modulus of a characteristic function of a random
	variable reaches $1$ at some point other than $0$ then the random
	variable has a lattice distribution.
	We have
	\[
	|\psi(\sigma+iy)|=\left|\int_0^\infty f(t)e^{2it\log
		t}e^{-(\sigma+it)t}\,dt\right|\leqslant \left|\int_0^\infty
	e^{-\sigma t}\,dt\right|\leqslant \frac{1}{\sigma},
	\]
	for $\sigma>0$.  Note that the estimate \ref{upper_bound_for_psi} for fixed $\sigma>0$ implies that $\psi(\sigma+iy)=O(1/|y|)$ as $|y|\to \infty$
 which means that the supremum of \(|\psi(\sigma+iy)|\) will be reached
on some finite point $y_0=y_0(\sigma)$. It remains to prove that this supremum cannot be equal to \(1/\sigma\). Indeed if
\[
|\psi(\sigma+iy_0)|= \frac{1}{\sigma},
\]
then recalling the definition of \(\psi\) we can rewrite this identity as
\[
\left|\int_0^\infty f(t)e^{2it\log t}e^{-\sigma t}
e^{-iy_0t}\,dt\right|=\int_0^\infty e^{-\sigma t}\,dt
\]
or equivalently
\[
e^{i\theta}\int_0^\infty f(t)e^{2it\log t}e^{-\sigma t}
e^{-iy_0t}\,dt=\int_0^\infty e^{-\sigma t}\,dt
\]
for some real $\theta$. Since $|f(t)|\leqslant 1$ taking the real  part
of the above equation  we have
\[
\Re\bigl(e^{i\theta}f(t)e^{2it\log t}e^{-iy_0t}\bigr)\equiv 1.
\]
The above identity together  with the fact that \(\bigl|e^{i\theta}f(t)e^{2it\log t}e^{-iy_0t}\bigr|\leqslant 1\) implies that \(\Im\bigl(e^{i\theta}f(t)e^{2it\log t}e^{-iy_0t}\bigr)\equiv 0\) and thus
\[
e^{i\theta}f(t)e^{it\log t}e^{-iy_0t}\equiv 1.
\]
Which means that
\[
\psi(s)=\int_0^\infty f(t)e^{2it\log t}e^{-s t}
\,dt=e^{-i\theta}\int_0^\infty e^{-s t} e^{iy_0t}\,dt=\frac{e^{-i\theta}}{s-iy_0}.
\]
However such function does not satisfy the equation
$-\psi'(s)=\psi^2(s-i)$.
\end{proof}
With the help of the just proven lemma we can obtain an upper bound
for $\psi(s)$ in the vicinity of the imaginary line $\Im(s)=0$.
\begin{lem} We have
\[
|\psi(s)|\leqslant \frac{1-\varepsilon}{1-|\Re(s)-1|(1-\varepsilon)},
\]
for $s$ belonging to the vertical strip
$-\frac{\varepsilon}{1-\varepsilon}<\Re(s)<\frac{2-\varepsilon}{1-\varepsilon}$,
where $\varepsilon$ is such that $\sup_{y\in
\mathbb{R}}|\psi(1+iy)|= 1-\varepsilon$.
\end{lem}
\begin{proof}
Applying the inequality of Lemma \ref{bound_sup} with $\sigma=1$ we
have
\[
\psi(1+iy)\leqslant 1-\varepsilon
\]
for all $y\in \mathbb{R}$ and some fixed $\varepsilon>0$. Hence
inequality of Lemma \ref{lem_ineq} yields
that
\begin{equation}
\label{psi_derivatives_bound}
\psi^{(k)}(1+iy)\leqslant
k!(1-\varepsilon)^{k+1}
\end{equation}
uniformly for $y\in \mathbb{R}$. This implies that $\psi(s)$ is
bounded in the vicinity  of the imaginary line $\Re(s)\geqslant
-\varepsilon'$ where $\varepsilon'<\varepsilon$. Indeed by Taylor
expansion
\[
\psi(s)=\sum_{k=0}^\infty \frac{\psi^{(k)}(1+iy)}{k!}(s-1-iy)^k
\]
Thus
\[
|\psi(s)|\leqslant \sum_{k=0}^\infty
(1-\varepsilon)^{k+1}|s-1-iy|^{k}=\frac{1-\varepsilon}{1-|s-1-iy|(1-\varepsilon)}
\]
for $|s-1-iy|<\frac{1}{1-\varepsilon}$. Suppose
$|\Re(s)-1|<\frac{1}{1-\varepsilon}$ then  taking  $y=\Im(s)$ we get
\[
|\psi(s)|\leqslant \frac{1-\varepsilon}{1-|\Re(s)-1|(1-\varepsilon)},
\]
for all $s$ lying in the strip $|\Re(s)-1|<\frac{1}{1-\varepsilon}$.
\end{proof}

 A more precise estimate can
be obtained combining the obtained two upper bounds for derivatives of $\psi(s)$.

\begin{lem}
We have an upper bound
\[
|\psi(s)|=O \left(\frac{1}{|s|}\right)
\]
in the region $\Re(s)>-\frac{\varepsilon'}{1-\varepsilon'}$. Where
$\varepsilon'$ is a fixed number that
$0<\varepsilon'<\varepsilon=1-\sup_{y\in \mathbb{R}}|\psi(1+iy)|$, the constant in the symbol  depends on $\varepsilon'$ only.
\end{lem}
\begin{proof} Putting $\sigma=1$ in our non-uniform bound (\ref{upper_bound_for_psi}) for
$\psi(s)$ we  have
\[
|\psi(1+iy)|\leqslant D/|y|
\]
for some fixed $D>0$. Again by induction  for $k\leqslant |y|/2$ we
have
\[
|\psi^{(k)}(1+iy)|\leqslant k!\left(\frac{2D}{|y|}\right)^{k+1}.
\]
Suppose $|\Re(s)-1|<\frac{1}{1-\varepsilon'}$. Let us take
$y=\Im(s)$. Combining the above upper bound with our previous
uniform estimate (\ref{psi_derivatives_bound}) for the derivatives
of $\psi^{(j)}(1+iy)$ we get
\[
\begin{split}
|\psi(s)|&\leqslant\sum_{k\leqslant |y|/2}
\left(\frac{2D}{|y|}\right)^{k+1}|s-1-iy|^k
+\sum_{k>|y|/2}|s-1-iy|^{k+1}(1-\varepsilon)^k
\\
&\leqslant \frac{2D}{|y|}\frac{1}{1-\frac{2D}{|y|}|s-1-iy|} +
\frac{|s-1-iy|\bigl(|s-1-iy|(1-\varepsilon)\bigr)^{|y|/2}}{1-|s-1-iy|(1-\varepsilon)}
\\
&\leqslant \frac{2D}{|y|-\frac{2D}{(1-\varepsilon')}} +
\frac{\bigl(\frac{1-\varepsilon}{1-\varepsilon'}\bigr)^{|y|/2}}{(1-\varepsilon')\left(1-\frac{1-\varepsilon'}{1-\varepsilon}\right)},
\end{split}
\]
for $|\Re(s)-1|<\frac{1}{1-\varepsilon'}$ and
$|y|>\frac{2D}{1-\varepsilon'}$. Since
$\frac{1-\varepsilon}{1-\varepsilon'}<1$ we have
\[
|\psi(s)|=O\left(\frac{1}{|s|}\right).
\]
\end{proof}
A number of conclusions can be drawn from the estimate of the just
proven lemma.

\begin{proof}[of Theorem \ref{main_thm}] 
The Laplace transform of \(tf(t)e^{2it\log t}\) is \(-\psi'(s)\) so, by inversion formula we have
\[
-f(t)e^{2it\log t} = \frac{1}{2\pi
i}\int_{\sigma-i\infty}^{\sigma+i\infty}\psi'(s)e^{ts}\,ds=
\frac{-1}{2\pi i
t}\int_{\sigma-i\infty}^{\sigma+i\infty}\psi^2(s-i)e^{ts}\,ds
\]
and taking into account that $|\psi(s-i)|\ll 1/|s| $ in the region
$\Re (s)\geqslant -2\eta$ for some fixed $\eta >0$ we can shift the
integration line to the left and obtain
\[
f(t)e^{2it\log t} = \frac{1}{2\pi i
t}\int_{-\eta-i\infty}^{-\eta+i\infty}\psi^2(s-i)e^{ts}\,ds\ll
e^{-\eta t}.
\]

\end{proof}

\begin{proof}[of Corollary \ref{cor}]
 The density is given by formula
\[
p(x)=\frac{1}{2\pi}\int_{-\infty}^\infty f(t)e^{-ixt}\,dt.
\]
	The fact that \(f(t)\) is exponentially decreasing \(|f(t)|\ll e^{-\eta |t|}\) at infinity \(|t|\to \infty\) immediately implies that the integral
	\[
	\frac{1}{2\pi}\int_{-\infty}^\infty f(t)e^{-ist}\,dt.
	\]
	is absolutely convergent in the vicinity of the real line \(|\Im(s)|< \eta\) where it defines an analytic function that coincides with the density of the quicksort distribution \(p(x)\) on the real line \(s=x\in \mathbb{R}\).
\end{proof}
\begin{cor}
The density function  $p(x)$ of the quicksort distribution can have
only finite number of zeros in any finite interval. The same is true for the derivatives of $p(x)$ of all
orders.
\end{cor}
\begin{proof}
 Since an analytic function that is not identically equal to zero  can have only finite number of zeros in any closed circle \(|s-x|\leqslant r/2\) for any \(x\in \mathbb{R}\), so the density \(p(x)\) can have only finite number of zeros in any finite interval \([x-r/2,x+r/2]\) with all \(x\in \mathbb{R}\).
\end{proof}

\acknowledgements The author sincerely thanks Prof. Hsien-Kuei Hwang for  numerous discussions on the topic of the paper as well as for his hospitality during the author's visits to Academia Sinica (Taiwan). The author also thanks the anonymous referee for  pointing out numerous errors in the previous draft of the paper and his suggestions that lead to considerable improvement in the exposition of the results  of the paper.
\bibliographystyle{abbrvnat}
%\bibliography{quicksort}

\end{document}